\documentclass[12pt]{amsart}

\usepackage{amsmath, amsthm, amssymb, amsfonts, amscd, color, enumerate, graphicx}
\usepackage{geometry}
\usepackage[shortlabels]{enumitem}
\usepackage{url}

\def\Z{{\mathbb Z}}
\def\F{{\mathbb F}}

\DeclareMathOperator{\lcm}{lcm}

\theoremstyle{plain}
\newtheorem{thm}{Theorem}

\newtheorem{cor}[thm]{Corollary}
\newtheorem{prop}[thm]{Proposition}
\newtheorem{lem}[thm]{Lemma}

\theoremstyle{definition}

\newtheorem{Remark}{Remark}

\newtheorem{Example}{Example}
\newtheorem{case}{Case}[thm]
\title{Orbits of Second Order Linear Recurrences over Finite Fields}

\author{Chatchawan Panraksa}
\address{Chatchawan Panraksa; Applied Mathematics Program; Mahidol University International College; Salaya \\
Nakhonpathom, 73170, Thailand.}
\email{chatchawan.pan@mahidol.edu}

\author{Naveen Somasunderam}
\address{Naveen Somasunderam; Department of Mathematics; State University of New York; Plattsburgh, NY, 12901, U.S.A.}

\email{nsoma001@plattsburgh.edu}

\date{August 18, 2024.}



\begin{document}

\maketitle 


\section{abstract}

Let $Q$ be the matrix $\displaystyle \begin{pmatrix} a & b \\ 1 & 0 \end{pmatrix}$ in $GL_2(\F_q)$ where $\F_q$ is a finite field, and let $G$ be the finite cyclic group generated by $Q$.  We consider the action of $G$ on the set $\F_q \times \F_q$. In particular, we study certain relationships between the lengths of the non-trivial orbits of $G$,  and their frequency of occurrence. This is done in part by investigating the order of elements of a product in an abelian group when the product has prime power order. For $q$ a prime and $b=1$, the orbits correspond to Fibonacci type linear recurrences modulo $q$ for different initial conditions. We also derive certain conditions under which the roots of the characteristic polynomial of $Q$ are generators of $\F_q^\times$. Examples are included to illustrate the theory.      
\section{Introduction}
\label{sec: Introduction}
Consider the matrix
\[
Q = \begin{pmatrix}
a & b \\
1 & 0
\end{pmatrix},
\]
over a finite commutative ring $R$ with $b$ a unit. Then $Q$ is invertible, and hence it generates a finite cyclic group $G$. We can consider the canonical action of $G$ on $R \times R$ given by
\begin{equation} 
\label{eqn: Action-of-G}
\begin{array}{cccc}
G: & R \times R    & \longrightarrow     & R \times R  \\
     & \begin{pmatrix}
        x_1 \\
        x_0
      \end{pmatrix}  &  \longmapsto          &  Q^n\begin{pmatrix}
        x_1 \\
        x_0
      \end{pmatrix}, 
\end{array}
\end{equation} 
for all $Q^n$ in $G$. Since any second order linear recurrence of the form $x_{n+2} = ax_{n+1} + bx_n$ in $R$
can be written as 
\begin{equation}
\label{eqn: matrix-form}
\begin{pmatrix}
        x_{n+1} \\
        x_n
\end{pmatrix} = Q^n \begin{pmatrix}
        x_1 \\
        x_0
      \end{pmatrix}, 
\end{equation} 
for initial conditions $x_0$ and $x_1$, the orbits of $G$ correspond to a linear recurrence in $R$ for different initial conditions. The periods of the sequence under different initial conditions correspond to the lengths of the orbits of $G$ in $R \times R$. 

Note that the set of zero initial conditions $x_0 = 0$, and $x_1 =0$ corresponds to a fixed point, and we shall call this the \textit{trivial} orbit. \textit{Hence, by the non-trivial orbits of $G$ we shall mean the orbits of $G$ associated to sets of non-zero initial conditions}. In this paper, we are concerned with how the lengths of the non-trivial orbits of $G$ are related when $R$ is any finite field $\F_q$.  In particular, when $R =\F_p$ for some prime $p$ and $b=1$ we have the case of a Fibonacci type sequence modulo $p$ whose periods are given by the lengths of the associated orbits of $G$ in $\F_p\times \F_p$.  

Linear recurrences over finite fields have been previously studied in the literature. See for example, Chapter $6$ of \cite{MR1294139} or \cite{selmer1966linear}. In \cite{MR1294139}, the authors study various relations between linear recurrences given certain divisibility relations between their characteristic polynomials. On the other hand, there is also an extensive literature on linear recurrences over $\Z/m\Z$ for some positive integer $m$. These studies primarily focus on how the periods are related modulo $p^{e+1}$ given the period mod $p^e$ and the distribution of the residues modulo $m$ (see for example \cite{MR364124}, \cite{MR120188}, \cite{MR0314750}, \cite{MR0314751} for an introduction). Classical techniques for studying these sequences involve analyzing the arithmetic properties of special functions such as the rank of apparition, while restricting the initial conditions to be $x_0 =0, x_1 =1$. The study of linear recurrences continues to be an active area of research with important applications (see for example \cite{MR1990179}, \cite{helleseth_sequences}). A recent work in \cite{MR4746536} for example, extends these classical techniques to study polynomial sequences over finite fields. 

Our work focuses on the periods of linear recurrences with a characteristic polynomial $p(x)$ in $\F_q[x]$ given  by 
\[  
p (x) = x^2 - ax - b, 
\]
whose associated matrix is $Q$, under different initial conditions. The analysis presented here depends on whether the characteristic polynomial $p(x)$  splits into distinct roots either over $\F_q$ or a quadratic extension $\F_{q^2}$, or whether it splits over $\F_q$ with repeated roots. 

Our main results are as follows. In the case when $p(x)$ has distinct roots over $\F_q$, the types of orbit lengths depend on the order of $-b$ in $\F_q^\times$. If $-b$ has prime order $r$, we show that the orbits lengths are either all equal to some $l$ where $r \mid l$ or else they are of lengths $l$ and $rl$ where $r \nmid l$. When $b =1$ in $\F_q$ with odd characteristic, this means that either all orbits have the same even length $l$, or there exists orbits of both an odd length $l$ and even length $2l$. 

We extend the above result to the case when $-b$ has a prime power order $r = p^\alpha$ for some prime $p$.  Then all orbits are of some equal length $l$, or of length $l$ and $p^{\alpha - v_p(l)} l$ where $v_p(l)$ is the p-adic valuation of $l$. We show this by first deriving a result relating the orders of elements $\gamma_1, \gamma_2$ and $\gamma_3$ of an abelian group that satisfy the relation $\gamma_1\gamma_2 = \gamma_3$ where the order of $\gamma_3$ is $p^\alpha$.  
If $-b$ is not of prime power order, we show how to construct a matrix $Q$ such that the non-trivial orbits of $G$ have three different length types. This analysis is done in Section \ref{sec: distinct-roots-case}.

In Section \ref{sec: Repeated-roots-case}, we consider the case when $p(x)$ has repeated roots over $\F_q$. We show that the orbit lengths are of the form $l$ and $pl$ for some $l$ where $p$ is the characteristic of the field $\F_q$. 

In each of the above cases, we also calculate the number of orbits of each length type and the total number of orbits, in terms of the smallest non-trivial orbit length $l$. 

In the case of distinct roots over $\F_q$, it is interesting to find conditions when a root of the characteristic polynomial $p(x)$ is a generator of $\F_q^\times$. When $a=b=1$, Shanks \cite{MR0297695} called such a root a Fibonacci primitive root. In \cite{MR1089524}, Phong generalized this notion and called a root of $p(x) = x^2 -ax -b$ to be a Lucas primitive root mod $p^e$ if it was a primitive root mod $p^e$. In this paper, by a Lucas primitive root or LPR we shall mean a root of $p(x) = x^2 - ax -1$ in $\F_q$ that is a generator of $\F_q^\times$. In Section \ref{sec: Lucas-Primitive-Roots}, we use the ideas developed previously to find certain criteria under which $p(x)$ has one or two LPRs. 

When $p(x)$ remains irreducible over $\F_q$, we show that there is only one non-trivial orbit length $l$. In this case, we give an upper bound on $l$ and a lower bound on the total number of non-trivial orbits. We show by some examples that these bounds are sharp. The techniques we use in this case are in the same vein as that presented in \cite{MR2910306}, where the authors are concerned with the periodicity of a sequence mod $p$ for a prime $p$, with initial conditions of $x_0 =0, x_1 =1$. On the other hand, our analysis is concerned with the orbits of $G$ over any finite field $\F_q$ as presented in Section \ref{sec: Irreducible-case}.

\section{Preliminaries}
\label{sec: preliminaries} 
Suppose that $p(x)$ has roots $\gamma_1$ and $\gamma_2$.  Then $p(x) = x^2 -ax -b = (x -\gamma_1)(x-\gamma_2)$ and comparing coefficients we obtain the relations
\begin{equation}
\label{eqn: coeff-formulas-1}
\gamma_1 + \gamma_2 = a,
\end{equation}
and 
\begin{equation}
\label{eqn: coeff-formulas-2}
\gamma_1 \gamma_2 = -b. 
\end{equation} 

In the case when $\gamma_1$ and $\gamma_2$ are distinct, we can consider the matrix $Q$ in diagonalized form either over $\F_q$ or $\F_{q^2}$.  Here, we can write $Q = PDP^{-1}$ where $D$ is a diagonal matrix. This is gives us the following lemma.
\begin{lem}
\label{lem: distinct-roots-possible-lengths}
Let $Q$ have characteristic polynomial $p(x)$ with distinct roots $\gamma_1, \gamma_2$ over $\F_q$ or $\F_{q^2}$. Then the orbit length of a non-zero initial vector $\displaystyle v = \begin{pmatrix}
x_{1} \\
x_{0}
\end{pmatrix} $  in $\F_q \times \F_q$ satisfies one or more of the following three conditions 
\begin{enumerate}[i.]
\item the length is equal to the order of $\gamma_1$, or  
\item the length is equal to the order of $\gamma_2$, or
\item the length is equal to the lcm of the orders of $\gamma_1$ and $\gamma_2$. 
\end{enumerate}
\end{lem}

\begin{proof}
From Equation \ref{eqn: matrix-form}, if the orbit length is $l$ then we have $PD^lP^{-1}v = v$ so that $P(D^l - I)P^{-1}v = 0$.  Multiplying by $P^{-1}$ we get
\begin{eqnarray}
\label{eqn: transformed-basis-eqn}
(D^l - I)P^{-1}v &=& \begin{pmatrix}
\gamma_1^l - 1 & 0 \\
0 & \gamma_2^l - 1 
\end{pmatrix} P^{-1}v \cr  
 &=& 0.
\end{eqnarray}
Since $v$ is non-zero, we have $P^{-1}v$ is not equal to zero. Therefore, Equation \ref{eqn: transformed-basis-eqn} is satisfied if and only if one of the following possibilities occur  
\begin{enumerate}[i.]
\item if the initial vector $v$ is such that $\displaystyle P^{-1}v = \begin{pmatrix}
c \\
0
\end{pmatrix}$ where $c \neq 0$, then the length $l$ is equal to the order of $\gamma_1$.
\item if the initial vector $v$ is such that $\displaystyle P^{-1}v = \begin{pmatrix}
0 \\
c
\end{pmatrix}$ where $c \neq 0$, then the length $l$ is equal to the order of $\gamma_2$.
\item if the initial vector $v$ is such that $\displaystyle P^{-1}v = \begin{pmatrix}
c_1 \\
c_2
\end{pmatrix}$ where $c_1, c_2$ are not equal to $0$, then $l$ is the lcm of the orders of $\gamma_1$ and $\gamma_2$. 
\end{enumerate}
\end{proof}  

To further analyze the relations between the possible orbit lengths when $p(x)$ splits into distinct roots, we must look at the orders of $\gamma_1$ and $\gamma_2$ in the group of units of $\F_q$ or $\F_{q^2}$. As such, we split our analysis of the distinct roots case into two parts. The first part corresponds to when $p(x)$ splits over $\F_q$, and this analysis is carried out in Section \ref{sec: distinct-roots-case}. 

A theorem of Wall (see Theorem 6 of \cite{MR120188}) states that the period of the Fibonacci sequence modulo $p$ divides $p-1$ when $p \equiv \pm 1 \mod 10$ i.e. when $p(x) = x^2 -x -1$ splits into distinct roots over $\F_p $. We can generalize this result to any second order sequence with arbitrary initial conditions over a finite field $\F_q$ using Lemma \ref{lem: distinct-roots-possible-lengths} as follows. 
\begin{cor}
\label{cor: length-divides-q-1}
Let $Q$ be such that its characteristic polynomial $p(x)$ splits into distinct roots over $\F_q$. Then the orbit lengths (and hence the periods of the corresponding sequences) divide $q-1$.     
\end{cor}

\begin{proof}
Suppose that the characteristic polynomial of $Q$ splits into distinct roots $\gamma_1$ and $\gamma_2$ over $\F_q$. By Lemma \ref{lem: distinct-roots-possible-lengths} the possible orbits lengths are of the form $|\gamma_1|$, $|\gamma_2|$ or $\lcm(|\gamma_1|,|\gamma_2|)$. Hence, the claim follows by Lagrange's Theorem. 
\end{proof}

We consider the case when $p(x)$ splits into repeated roots over $\F_q$ in Section \ref{sec: Repeated-roots-case}. Note that any diagonalizable $2 \times 2$ matrix with repeated eigenvalues must be a scalar multiple of the identity. Hence, in this case $Q$ is not diagonalizable. Since we assume that the characteristic polynomial splits in $\F_q$ we can consider the Jordan form of $Q$ over $\mathbb{F}_q$ given by $JDJ^{-1}$ where 
\begin{equation} 
\label{eqn: Jordan-matrix}
D = \begin{pmatrix}
\gamma & 1 \\
0 & \gamma 
\end{pmatrix}, 
\end{equation} 
$\gamma$ is a repeated root of $p(x)$ and $J$ is an invertible matrix in $M_{2\times 2}(\mathbb{F}_q)$ (see for example \cite{freidberg2002linear}, Chapter 7 for the Jordan form over an arbitrary field $\F$). It is easily shown by induction that 
\[
D^n = \begin{pmatrix}
\gamma^n & n\gamma^{n-1} \\
0 & \gamma^n 
\end{pmatrix}. 
\]
We use the Jordan form $JDJ^{-1}$ with $D$ given by Equation \ref{eqn: Jordan-matrix} to classify the orbit length relationships in the case of repeated roots. In section \ref{sec: Repeated-roots-case} we will show that when $Q$ has repeated roots over $\F_q$, not all possible orbit lengths will divide $q-1$. Hence, Corollary \ref{cor: length-divides-q-1} does not extend to the repeated roots case.

When $p(x)$ remains irreducible over $\F_q$ we must consider a quadratic extension $\F_{q^2}$ over which $p(x)$ splits, and this analysis is done in  Section \ref{sec: Irreducible-case}.

\section{Distinct roots over $\F_q$}
\label{sec: distinct-roots-case}

In this section, we consider the case when the characteristic polynomial $p(x)$ of $Q$ splits  over $\F_q$ with distinct roots $\gamma_1$ and $\gamma_2$. 

Using the expression $\gamma_1 \gamma_2 = -b$ as given by Equation \ref{eqn: coeff-formulas-2}, our goal is to find any relationships between the orders of $\gamma_1$ and $\gamma_2$ in $\F_q^\times$ given information on the order of $-b$. We do this in the more general setting of an abelian group $F$, with elements   $\gamma_1$, $\gamma_2$ and $\gamma_3$ related by the expression $\gamma_1 \gamma_2 = \gamma_3$. In particular, we show that if $\gamma_3$ has prime or prime power order then the orders of $\gamma_1$ and $\gamma_2$ have certain divisibility relations as given by Lemma \ref{lem: main-residue-case} and Theorem \ref{thm: prime-power-order}. Then considering $-b$ to play the role of $\gamma_3$, we derive certain relations between the possible orbit lengths of the action of $G$ on $\F_q \times \F_q$ as stated in Proposition \ref{prop: r-divisibility-cond} and Theorem \ref{thm: main-residue-case}.

We use the following known facts about the order of elements in a group 
\begin{enumerate}[(1)]
\item If $\gamma_1 \gamma_2 = \gamma_3$ and $\gcd(|\gamma_1|, |\gamma_2|) =1$, then $|\gamma_3|= |\gamma_1||\gamma_2|$. 
\item For any positive integer $k$, we have $\displaystyle 
|\gamma_3^k| = \frac{|\gamma_3|}{\gcd(k,|\gamma_3|)}.  
$
\end{enumerate}

We have the following key Lemma. 
\begin{lem}
\label{lem: main-residue-case}
Let $F$ be an abelian group and $\gamma_1 \gamma_2 = \gamma_3$, with $m = |\gamma_1|, n=|\gamma_2|, r=|\gamma_3|$. Assume without loss of generality that $m \leq n$. Then the following hold
\begin{enumerate}[(a)]
\item If $r \mid m,$ then $m=n$. 
\item If $\gcd(r,m) =1$, then $n = rm$. 
\item $\gcd(n,r) > 1$, provided that $r \neq 1$. 
\end{enumerate}
In particular, if $r$ is a prime then either $m=n$ or $n=rm$. 
\end{lem}

\begin{proof}\mbox{} 
\begin{enumerate}[(a)]
\item To show part (a), if $r \mid m$ we have $\displaystyle \gamma_2^m = \frac{\gamma_3^m}{\gamma_1^m} = \gamma_3^m =1$. Hence, $n \mid m$ and since $m \leq n$ we have $m=n$. 
\item We have $\displaystyle \gamma_2 = \gamma_3\gamma_1^{-1}$. Since $|\gamma_1^{-1}| = |\gamma_1| = m$, if $r$ and $m$ are relatively prime then $\displaystyle n = rm$. 

\item Assume $\gcd(n,r) =1$. Note that we have  $\displaystyle \gamma_2^{rm} = \frac{\gamma_3^{rm}}{\gamma_1^{rm}} = 1$, so that $n \mid rm$. Hence, if $\gcd(n,r) =1$ then $n \mid m$ by Euclid's lemma. Since $m \leq n$, we would have to conclude that $n=m$. Hence $\displaystyle 1 = \gamma_2^m =  \frac{\gamma_3^m}{\gamma_1^m} = \gamma_3^m$, so that $r \mid m = n$, a contradiction. 
\end{enumerate}
In the case when $r$ is prime, either $r \mid m$ or $\gcd(r,m) =1$. Hence, either $m=n$ or $n = rm$. 
\end{proof}


Of particular interest is when $F = \F_q^\times$ and $\gamma_3$ is $-b$. We can use Lemmas \ref{lem: distinct-roots-possible-lengths} and \ref{lem: main-residue-case} to find the following relations between the orbit lengths.

\begin{prop} 
\label{prop: r-divisibility-cond}
Let $\F_q$ be a finite field, and let $\displaystyle Q = \begin{pmatrix} a & b \\ 1 & 0 \end{pmatrix}$ be such that its characteristic polynomial $p(x)$ splits into distinct roots $\gamma_1$ and $\gamma_2$ over $\F_q$ with $r = |-b|$, $m = |\gamma_1|$, and $n = |\gamma_2|$ in $\F_q^\times$. Without loss of generality assume that $m \leq n$. Then the lengths of the non-trivial orbits of $G= \langle Q \rangle$ under its canonical action on $\F_q \times \F_q$ satisfy the following conditions
\begin{enumerate}[(a)]
\item If $r \mid m$ then all orbits have the same length $m$. 
\item If $\gcd(r,m) =1$ then there are two types of orbits lengths $m$ and $rm$. 
\end{enumerate}
\end{prop}  

\begin{proof}
 Applying Lemma \ref{lem: main-residue-case} we have $m = n$ if $r \mid m$, and if $\gcd(r,m) = 1$ we have $n = rm$. Using Lemma \ref{lem: distinct-roots-possible-lengths} if $r \mid m$ all orbits are of equal length $m$, and if $\gcd(r,m) = 1$ we have $\lcm(n,m) = n$ so that there are orbits of length $m$ and $rm$. 
\end{proof}

When the order $r$ of $-b$ is a prime, then the orbit relations stated in Proposition \ref{prop: r-divisibility-cond} are in fact the only possibilities. Note that this covers the interesting case when $b=1$ and $\F_q$ is of odd characteristic, since in that case $r$ would be equal to $2$. We state these results in Corollary \ref{cor: r-is-prime}.

\begin{cor}
\label{cor: r-is-prime}
If $r = |-b|$ is a prime then the lengths of the non-trivial orbits are of two possible types 
\begin{enumerate}[(a)]
\item all orbits have the same length $m$ where $r \mid m$, or 
\item there are orbits of length $m$ and length $rm$, where $r \nmid m$. 
\end{enumerate}
In particular, if $b=1$ and $\F_q$ is of odd characteristic then 
\begin{enumerate}[resume*]
\item all orbits have the same length $m$ where $m$ is even, or 
\item there are orbits of lengths $m$ and $2m$, where $m$ is odd.  
\end{enumerate}
\end{cor}

\begin{proof}
Since $r$ is prime, either $r \mid m$ or $\gcd(r,m) =1$. Hence, applying Proposition \ref{prop: r-divisibility-cond} we get the desired results. Parts (c) and (d) follow by taking $r=2$.     
\end{proof}

\begin{Remark}
Note that Corollary \ref{cor: r-is-prime} does not cover the case when $Q$ has distinct roots with $b=1$ over $\F_q$ in characteristic $2$. In that case, we have $-b = 1$ so that $|-b|$ divides $m$. Hence, by Lemma \ref{lem: distinct-roots-possible-lengths} and part (a) of Lemma \ref{lem: main-residue-case}  we would have $m =n$ and all non-trivial orbits are of equal length. A similar result also holds if $b=-1$ over $\F_q$ in odd characteristic.  
\end{Remark}

It is known in the literature that in the case of $p(x) = x^2 -x -1$ having distinct roots over $\F_p$ or a suitable extension field, the special case of parts (c) and (d) of Corollary \ref{cor: r-is-prime} hold (for example, see Theorem 2.5 of \cite{MR1188731}). 


We now look at a more general case when the order of $-b$ is a prime power, and show that the non-trivial orbits of $G$ are at most of two different lengths as given in Theorem \ref{thm: main-residue-case}. The proof of Theorem \ref{thm: main-residue-case} requires an extension of Lemma \ref{lem: main-residue-case}, which relates the orders of elements $\gamma_1, \gamma_2$ and $\gamma_3$ of an abelian group $F$ where $\gamma_1\gamma_2 = \gamma_3$ and $\gamma_3$ has a prime power order. This is done in Theorem \ref{thm: prime-power-order}.
Before proving Theorem \ref{thm: prime-power-order}, we state and prove the following two required lemmas. Recall that the \textit{$p$-adic valuation $v_p(m)$} of an integer $m$ is the largest integer $k$ such that $p^k$ divides $m$. 
\begin{lem}
\label{lem: p-adic-valuation-power} 
Let $H$ be an abelian group and $a$ be in $H$ with $m= |a|$. Then for any prime $p$ and any positive integer $s$, we have $p \mid |a^s|$ if and only if $v_p(m) > v_p(s)$.    
\end{lem} 

\begin{proof}
Since $\displaystyle |a^s| = \frac{m}{\gcd(m,s)}$, we have 
\[
v_p(m) = v_p(|a^s|) + v_p(\gcd(m,s)). 
\]
To show the forward direction, if $p \mid |a^s|$ then $v_p(|a^s|) \geq 1$ so that 
\begin{eqnarray*}
 v_p(m) &\geq&  1 + v_p(\gcd(m,s)), \cr
        & = & 1 + \min(v_p(m), v_p(s)). 
\end{eqnarray*}
If $\min(v_p(m), v_p(s)) = v_p(m)$ we get $v_p(m) \geq  1 + v_p(m)$, a contradiction. Therefore, $\min(v_p(m), v_p(s)) = v_p(s)$ and hence $v_p(m) \geq  1 + v_p(s)$. 

To show the reverse direction, assume that $v_p(m) > v_p(s)$. Then $\min(v_p(m), v_p(s)) = v_p(s)$, and hence  
\begin{eqnarray*}
 v_p(m) &=&  v_p(|a^s|) + \min(v_p(m), v_p(s)), \cr
        &=&  v_p(|a^s|) + v_p(s),
\end{eqnarray*}
and so 
\begin{eqnarray*}
 v_p(|a^s|) &=&  v_p(m) - v_p(s), \cr
            &\geq& v_p(s) + 1 - v_p(s), \cr
            &=& 1. 
\end{eqnarray*}
In this case, $p \mid |a^s|$ as claimed. 
\end{proof}

\begin{lem}
\label{lem: main-residue-ps-case} 
Let $F$ be an abelian group, and let $\gamma_1 \gamma_2 = \gamma_3$ in $F$ with $m = |\gamma_1|, n=|\gamma_2|$. Suppose that $r=|\gamma_3|$ is of the form $ps$ where $p$ is a prime and $s$ is a positive integer. Without loss of generality, assume that $v_p(m) \leq v_p(n)$. Then
\begin{equation*}
n = \left\{ \begin{array}{ccc}
             \frac{\gcd(n,s)}{\gcd(m,s)} m  & \, & \text{if } v_p(m) > v_p(s), \\ 
                                            &    & \\
             p\frac{\gcd(n,s)}{\gcd(m,s)} m  & \, & \text{if } v_p(m) \leq v_p(s).\\ 
             \end{array} 
     \right. 
\end{equation*}
\end{lem}
\begin{proof}
Since $\gamma_1 \gamma_2 = \gamma_3$ we have $\gamma_1^s \gamma_2^s = \gamma_3^s$, where $|\gamma_3^s| = p$ a prime. We look at two cases. 

\begin{case} $|\gamma_1^s| \leq |\gamma_2^s|$ \\
Applying Lemma \ref{lem: main-residue-case}, we have that either $|\gamma_1^s| = |\gamma_2^s|$ if $p \mid |\gamma_1^s|$ or $|\gamma_2^s| = p|\gamma_1^s|$ if $p \nmid |\gamma_1^s|$. We have $\displaystyle |\gamma_1^s| = \frac{m}{\gcd(m,s)}$, and $\displaystyle |\gamma_2^s|= \frac{n}{\gcd(n,s)}$. Therefore, if $p \mid |\gamma_1^s|$ then
\[
n = \frac{\gcd(n,s)}{\gcd(m,s)} m,  
\]
and if $p \nmid |\gamma_1^s|$  then 
\[
n = p\frac{\gcd(n,s)}{\gcd(m,s)} m. 
\]
Note that using Lemma \ref{lem: p-adic-valuation-power}, the condition $p \mid |\gamma_1^s|$ can be replaced by $v_p(m) > v_p(s)$ and the condition $p \nmid |\gamma_1^s|$ can be replaced by $v_p(m) \leq v_p(s)$.  
\end{case} 

\begin{case} $|\gamma_1^s| > |\gamma_2^s|$ \\
We show that this case is not possible. Applying Lemma \ref{lem: main-residue-case}, if $p \mid |\gamma_2^s|$ we get $|\gamma_2^s| = |\gamma_1^s|$ which is not possible by our assumption for this case. Therefore, we will assume $p \nmid |\gamma_2^s|$. Then, using Lemma \ref{lem: main-residue-case} we have $|\gamma_1^s| = p |\gamma_2^s|$ or  
\[
m = p\gcd(m,s) |\gamma_2^s|. 
\]
Hence, $v_p(m) = 1 + v_p(\gcd(m,s))$ with the $p$-adic valuation of $|\gamma_2^s| = 0$ since $p \nmid |\gamma_2^s|$. We claim that $v_p(\gcd(m,s)) = \min(v_p(m),v_p(s)) = v_p(m)$. To see this, note that  $n = |\gamma_2| = |\gamma_2^s|\gcd(n,s) = \gcd(n,s)$ since $p \nmid |\gamma_2^s|$. So, $v_p(n) = v_p(\gcd(n,s)) = \min(v_p(n), v_p(s))$ i.e. $v_p(n) \leq v_p(s)$. Then since $v_p(m) \leq v_p(n)$ by assumption of the Lemma, we have $\min(v_p(m), v_p(s)) = v_p(m)$. This proves our claim that $v_p(\gcd(m,s)) = \min(v_p(m),v_p(s)) = v_p(m)$. This then gives us $v_p(m) = 1 + v_p(\gcd(m,s)) = 1 + v_p(m)$, a contradiction.     
\end{case}
\end{proof}

\begin{thm}
\label{thm: prime-power-order}
Let $F$ be an abelian group, and $\gamma_1 \gamma_2 = \gamma_3$ with $r=|\gamma_3|$. Assume that $r=p^\alpha$ where $p$ is a prime number and $\alpha$ is a positive integer. Let $m= |\gamma_1|$, $n= |\gamma_2|$, $k=\nu_p(m)$ and suppose without loss of generality that $v_p(m) \leq v_p(n)$. Then the following hold
\begin{equation*}
n = \left\{ \begin{array}{ccc}
     m  & \, & \text{if } v_p(m) \geq \alpha, \\ 
                                            &    & \\
             p^{\alpha-k}m  & \, & \text{if } v_p(m) < \alpha.\\ 
             \end{array} 
     \right. 
\end{equation*}
\end{thm}
\begin{proof}
Using Lemma \ref{lem: main-residue-ps-case} with $r=p^\alpha$ and $s=p^{\alpha-1}$, we have  
\begin{equation}
\label{eqn: prime-power-order}
n = \left\{ \begin{array}{ccc}
             \dfrac{\gcd(n,p^{\alpha-1})}{\gcd(m,p^{\alpha-1})}m  & \, & \text{if } v_p(m) \geq \alpha \\ 
                                            &    & \\
             p\dfrac{\gcd(n,p^{\alpha-1})}{\gcd(m,p^{\alpha-1})}m  & \, & \text{if } v_p(m) < \alpha.\\ 
             \end{array} 
     \right. 
\end{equation}

Based on Equation \ref{eqn: prime-power-order} we split our analysis into two cases, when $ v_p(m) \geq \alpha$ and $ v_p(m) < \alpha$.

\begin{case} $ v_p(m) \geq \alpha$. \\
Using Equation \ref{eqn: prime-power-order}, we have 
\begin{equation} 
\label{eqn2: prime-power-order}
\begin{array}{cc} 
n &= \dfrac{\gcd(n,p^{\alpha-1})}{\gcd(m,p^{\alpha-1})}m, \\ 
  & \\ 
  &= \dfrac{\gcd(n,p^{\alpha-1})}{p^{\alpha-1}}m,
\end{array}
\end{equation} 
where the second line follows from the fact that $p^\alpha$ divides $m$. Since $v_p(n) \geq v_p(m)$ by assumption, we have $v_p(n) \geq \alpha$. Hence, $\gcd(n,p^{\alpha-1}) = p^{\alpha-1}$ and 
\begin{align*}
n & = \dfrac{\gcd(n,p^{\alpha-1})}{p^{\alpha-1}}m\\
  & = m. 
\end{align*}
\end{case}

\begin{case} $v_p(m) < \alpha$. 
Using Equation \ref{eqn: prime-power-order} we have 
\[ 
n = p \dfrac{(n,p^{\alpha-1})}{(m,p^{\alpha-1})}m, 
\]   
from which we get 
\[
(m,p^{\alpha-1}) n = p (n,p^{\alpha-1}) m. 
\]
Taking the $p$-adic valuation of both sides of the above equation we get 

\[
\min(v_p(m), \alpha-1) + v_p(n) = 1 + \min(v_p(n),\alpha-1) + v_p(m). 
\]
And since $\min(v_p(m),\alpha-1) = v_p(m)$, this reduces to 
\[
v_p(n) = 1 + \min(v_p(n),\alpha-1). 
\]
If $\min(v_p(n),\alpha-1) = v_p(n)$ then $v_p(n) = 1 + v_p(n)$, a contradiction. Therefore, $\min(v_p(n),\alpha-1) = \alpha-1$ from which we conclude that $v_p(n) = \alpha$. Hence, from Equation \ref{eqn: prime-power-order} with $v_p(m) = k$ we get 

\begin{align*}
n &= p \dfrac{(n,p^{\alpha-1})}{(m,p^{\alpha-1})}m\\
  &= p\dfrac{p^{\alpha-1}}{p^k} m \\
  &= p^{\alpha -k} m.
\end{align*}
\end{case}
\end{proof}

Using Lemma \ref{lem: distinct-roots-possible-lengths} and Theorem \ref{thm: prime-power-order}, when the order $r$ of $-b$ is a prime power i.e. $ r = p^\alpha$ for some positive integer $\alpha$ we can guarantee the existence of at most two types of non-trivial orbit lengths. We state this in the next theorem. 

\begin{thm}
\label{thm: main-residue-case}
Let $\F_q$ be a finite field, and let $\displaystyle Q = \begin{pmatrix} a & b \\ 1 & 0 \end{pmatrix}$ be such that its characteristic polynomial $p(x) = x^2 -ax -b$ splits into distinct roots over $\F_q$. Further, suppose the order of $|-b|$ is of the form $p^\alpha$ for some prime $p$ and positive integer $\alpha$. Then the lengths of the non-trivial orbits of $G= \langle Q \rangle$ under its canonical action on $\F_q \times \F_q$ are of two possible types 
\begin{enumerate}[(a)]
\item All orbits have the same length $l$, where $v_p(l) \geq \alpha$, or 
\item There are orbits of length $l$ and length $p^{\alpha-k}l$, where $k = v_p(l) < \alpha $. 
\end{enumerate}
\end{thm}

\begin{proof}
Let $\gamma_1$ and $\gamma_2$ be the distinct roots of $p(x)$ in $\F_q$ with $\gamma_1 \gamma_2 = -b$, where $|-b| = p^\alpha$. Applying Theorem \ref{thm: prime-power-order}, 
we have two possibilities for the orders of $\gamma_1$ and $\gamma_2$. Either $|\gamma_1| = |\gamma_2|$ or $|\gamma_2| = p^{\alpha-k} |\gamma_1|$ where wlog we assume $v_p(|\gamma_1|) = k < \alpha$. 

Hence, $\lcm(|\gamma_1|,|\gamma_2|)$ is either equal to $|\gamma_1|$ or $p^{\alpha-k}|\gamma_1|$. Now, applying Lemma \ref{lem: distinct-roots-possible-lengths} we get the desired results. 
\end{proof}

The following theorem gives us the number of orbits of each type. 

\begin{thm}[Number of orbits of each type]
\label{thm: distinct-roots-orbit-numbers}
Let $Q$ be such that its characteristic polynomial has distinct roots over $\F_q$, and suppose there are only non-trivial orbits of lengths $l$ and $kl$ (with $k$ possibly equal to one). Then there are $\displaystyle \frac{q-1}{l}$ orbits of length $l$ and $\displaystyle \frac{q(q-1)}{kl}$ orbits of length $kl$. The total number of non-trivial orbits is equal to $\displaystyle \frac{(q-1)(q+k)}{kl}$. 

\end{thm}

\begin{proof}
Suppose that $\gamma_1, \gamma_2$  are the distinct roots of $Q$. Since there are only two possible orbit lengths $l$ and $kl$, we must have $|\gamma_1| = l$, and $|\gamma_2| = k|\gamma_1| = kl$. 

Let $n_1$ be the number of orbits of length $l$, and $n_2$ be the number of orbits of length $kl$. Note that an orbit of length $l$ only occurs only when the initial vector in the transformed basis as given by Equation \ref{eqn: transformed-basis-eqn} is of the form $\displaystyle \begin{pmatrix} x \\ 0 \end{pmatrix}$. Moreover, the orbit of such a vector in the transformed basis is of the form
\[
\begin{array}{cccc}
\begin{pmatrix} x \\ 0 \end{pmatrix}, & \begin{pmatrix} \gamma_1 x \\ 0 \end{pmatrix}, & \cdots , & \begin{pmatrix} \gamma_1^{l-1}x \\ 0 \end{pmatrix}.
\end{array}
\] 
Since there are $q-1$ such vectors in $\F_q \times \F_q$, we conclude that the number of such orbits is equal to
\[
n_1 = \frac{q-1}{l}.
\]
Then, since the orbits partition the set $\F_q \times \F_q$ we get
\[
n_1l + n_2kl = q^2 -1,
\]
where we subtract one to account for the trivial orbit. So,  
\begin{eqnarray*}
n_2k &=&  \frac{q^2-1}{l} - n_1, \cr
     &=&  \frac{q(q-1)}{l}
\end{eqnarray*}
and 
\[
n_2 = \frac{q(q-1)}{kl}. 
\]
The total number of non-trivial orbits is 
\begin{eqnarray*}
n_1 + n_2 &=& \frac{(q-1)(q+k)}{kl}.
\end{eqnarray*}
\end{proof} 

\begin{cor}
Let $Q$ be such that its characteristic polynomial has distinct roots over $\F_q$. Then, if all the non-trivial orbits of $Q$ are of equal length $l$, or of two different lengths $l$ and $kl$, the total number of non-trivial orbits is greater than or equal to $q+1$. 
\end{cor}

\begin{proof}
In the case of all orbits having equal length $l$, since the largest possible value of $l$ in Theorem \ref{thm: distinct-roots-orbit-numbers} is $l = q-1$ the number of orbits has to be greater than or equal to $q+1$. In the case of two types of orbit lengths, the largest possible orbit length is $kl = q-1$ and hence by Theorem \ref{thm: distinct-roots-orbit-numbers} the number of orbits is greater than or equal to $q+k$. In either case, the number of orbits is greater than or equal to $\min(q+1, q+k) = q+1$.
\end{proof}


Recall that by a primitive root, we mean a root of the characteristic polynomial of $Q$ that is a generator to $\F_q^\times$. We can classify when a  primitive root occurs by looking at the number of non-trivial orbits, as given by the next Corollary. 
\begin{cor}
\label{cor: primitive-roots-prime-power}
Let $Q$ be such that its characteristic polynomial has distinct roots over $\F_q$, and $-b$  is of a prime power order. Then 
\begin{enumerate}[(a)]
\item $Q$ has two primitive roots (i.e. generators of $\F_q^\times$) if and only if there are exactly $q+1$ non-trivial orbits. 
\item $Q$ has exactly one primitive root (i.e. a generator of $\F_q^\times$) if and only if  there are exactly $q$ non-trivial orbits of length $q-1$.   
\end{enumerate}
\end{cor}

\begin{proof}
The claims follow from Theorem \ref{thm: main-residue-case} and the expressions for the number of orbits of different lengths as given by Theorem \ref{thm: distinct-roots-orbit-numbers}. 
\begin{enumerate}[(a)]
\item Suppose $Q$ has two primitive roots (i.e. generators of $\F_q^\times$). Then there is only one non-trivial orbit length of $l = q-1$. Hence the number of such orbits is $\frac{q^2-1}{q-1} = q+1$. On the other hand, if there are exactly $q+1$ non-trivial orbits then 
\[
q+1 = \frac{(q-1)(q+k)}{kl},  
\]
using the formula given by Theorem \ref{thm: distinct-roots-orbit-numbers}. From this we get, 
\begin{eqnarray}
\label{eqn: primitive-roots-prime-power}
    (q-1)(q+k) &=& (q+1)kl, \cr 
               &\leq&  (q+1)(q-1), 
\end{eqnarray}
where the last line on Equation \ref{eqn: primitive-roots-prime-power} follows form the fact that the largest orbit length of $kl$ has to be bounded by $q-1$. From this, we get $k = 1$. Hence, both distinct roots of $Q$ have the same order $l$ and $l = q-1$. 

\item Suppose $Q$ has exactly one primitive root (i.e. a generator of $\F_q^\times$). Then, the order of that root is $q-1$. Now, using the formula given by Theorem \ref{thm: distinct-roots-orbit-numbers} we get $n_2 = q$ orbits of length $q-1$. 

On the other hand, suppose there are exactly $q$ orbits of length $q-1$. Then by Theorem \ref{thm: main-residue-case}, the possible orbit lengths are $l$ and $kl$ where $kl= q-1$. If $k=1$, we would have a total of $q$ orbits of all equal length $q-1$, which is not possible since $q(q-1) = q^2 -q$ is less than cardinality of $\F_q \times \F_q$. Therefore, $k >1$. In that case, $l < kl = q-1$ and $Q$ has a root that is not primitive. We conclude that $Q$ has exactly one primitive root.   
\end{enumerate}
\end{proof}

\begin{Example}
\label{ex: permutation-odd-char}
Consider the case when $Q$ is the permutation matrix $\displaystyle \begin{pmatrix} 0 & 1 \\ 1 & 0 \end{pmatrix}$ over $\F_q$ of odd characteristic. We have $a=0$, $b=1$ and distinct roots $1$ and $-1$ . There are $q-1$ non-trivial orbits of length $1$ and $\displaystyle \frac{q(q-1)}{2}$ non-trivial orbits of length $2$.      
\end{Example} 


\begin{Example}  
Consider the finite field $\F_{163}$. Let $a = 9$ and $b = 159$. We have $-b = 4$ and $|-b| = 81 = 3^4$. The characteristic polynomial $p(x) = x^2 - 9x - 159$ has two distinct roots $\gamma_1 = 23$ and $\gamma_2 = 149$. Using a computer simulation, we find that we have non-trivial orbits of length $l = 18 = 2 \cdot 3^2$, and of length $p^{\alpha - v_p(l)} \cdot l = 3^{4-2} \cdot 18 =162$. Using Theorem \ref{thm: distinct-roots-orbit-numbers} there are $\displaystyle \frac{q-1}{l} = \frac{162}{18} = 9$ orbits of length $18$, and $\displaystyle \frac{q(q-1)}{kl} = \frac{163 \cdot 162}{162} = 163$ orbits of length $162$.
\end{Example}


A more general extension of Theorem \ref{thm: prime-power-order} is not possible. That is, given an element $\gamma_3$ of some finite abelian group whose order is not a prime power, we can always pick elements $\gamma_1$ and $\gamma_2$ such that  $\gamma_1\gamma_2 = \gamma_3$, and the orders of $\gamma_1$ and $\gamma_2$ are relatively prime. This is proved in Theorem \ref{thm: order-not-prime-power}. In particular, this means that given a $b$ such that $-b$ has a non-prime power order, there exists an $a$ such that $G = \langle Q \rangle$ has non-trivial orbits of three different lengths.   

\begin{thm}
\label{thm: order-not-prime-power}
Let $F$ be a finite abelian group. Let $\gamma_3\in F$ be such that $|\gamma_3|=r=mn$ and $\gcd(m,n)=1.$ Then there exist elements $\gamma_1$ and $\gamma_2$ in $F$ such that 
$\gamma_1 \gamma_2 = \gamma_3$ and $|\gamma_1|=m, |\gamma_2|=n$.   
\end{thm}

\begin{proof}
Let $g= \gamma_3$ and $\langle g \rangle$ the subgroup generated by $g$ with order $|g|= r$. Let $\gamma_1 = g^{k_1}$ where $k_1$ satisfies
\begin{align*}
k_1\equiv 0 &\mod n, \\
k_1 \equiv 1 & \mod m.
\end{align*}
The Chinese Remainder Theorem guarantees there is a unique such $k_1 \mod mn.$ \\
We claim that $|\gamma_1|=|g^{k_1}|=m$.

To see this, note that $|\gamma_1|=\dfrac{r}{\gcd(k_1,r)}.$ Since $k_1\equiv 0\mod n,$ we have $k_1=nt$ for some $t$. Moreover, since $k_1\equiv 1\mod m$ we have $nt \equiv 1\mod m$ i.e. $\gcd(t,m)=1.$ Therefore, 
\begin{eqnarray*}
|\gamma_1| &=& \dfrac{r}{\gcd(k_1,r)}, \cr 
           & & \cr
           &=&\dfrac{mn}{\gcd(nt,mn)}, \cr 
           & & \cr
           &=&\dfrac{m}{\gcd(t,m)}, \cr 
           & & \cr
           &=& m.
\end{eqnarray*}
Now choose $k_2 = r+1-k_1$, and  $\gamma_2=g^{k_2}$. Note that this choice of $\gamma_2$ satisfies the requirement $\gamma_1 \gamma_2 = \gamma_3$. We claim that $|\gamma_2|=n$. To see this, note that the order of $\gamma_2$ is 
\begin{eqnarray*} 
|\gamma_2| &=&\dfrac{r}{\gcd(r+1-k_1,r)}, \cr 
           & & \cr 
           &=& \dfrac{r}{\gcd(1-k_1,r)}.
\end{eqnarray*}
Using $k_1=nt$, we have $|\gamma_2|=\dfrac{mn}{\gcd(1-nt,mn)}$. Since $k_1 = nt\equiv 1 \mod m$, we have $m \mid 1-nt$. Moreover, note that $\gcd(1-nt, n)=1.$ Therefore, $\gcd(1-nt, mn)= m$. Thus, $|\gamma_2|=\dfrac{mn}{m}=n.$
\end{proof}

Using Theorem \ref{thm: order-not-prime-power}, given a $-b$ of non-prime power order we can pick an $a$ so that $G = \langle Q \rangle$ has three non-trivial orbit length types. We state this as the next corollary. 

\begin{cor}
Let $b$ in $\F_q^\times$ be such that $-b$ has non-prime power order. Then there exists an $a$ in $\F_q^\times$ such that $\displaystyle Q = \begin{pmatrix} a & b \\ 1 & 0 \end{pmatrix}$, and $G = \langle Q \rangle$ under its canonical action on $\F_q \times F_q$ has non-trivial orbits of three different lengths.       
\end{cor}

\begin{proof}
Choose $\gamma_3 = -b$. Using Theorem \ref{thm: order-not-prime-power}, we can find $\gamma_1$ and $\gamma_2$ with relatively prime orders. Now pick $a = \gamma_1 + \gamma_2$. For this choice of $a$ and $b$, $G$ will have the desired property.  
\end{proof}

\section{Repeated roots over $\F_q$}
\label{sec: Repeated-roots-case}

In this section, we look at the case when $Q$ has repeated roots. As discussed in Section \ref{sec: preliminaries}, we need to consider the Jordan form of $Q$ over $\mathbb{F}_q$ given by $JDJ^{-1}$ where 
\[
D = \begin{pmatrix}
\gamma & 1 \\
0 & \gamma 
\end{pmatrix}, 
\]
and $\gamma$ is a repeated root of $p(x)$. Using the Jordan form, we show that if $\F_q$ is of characteristic $p$ then there are non-trivial orbits of length $l$ and $pl$ where $l$ is the order of the root $\gamma$ in $\F_q^\times$.  This is stated as Theorem \ref{thm: repeated-roots-orbit-lengths}. In Theorem \ref{thm: repeated-roots-number-of-orbits}, we calculate the number of orbits of each type.

\begin{thm}
\label{thm: repeated-roots-orbit-lengths}
Let $\F_q$ be a field of characteristic of $p$. Suppose that the characteristic polynomial of $Q$ has a repeated root $\gamma$ over $\F_q$. Then the non-trivial orbits of $G$ are of length $l$ and $pl$ where $l$ is the order of $\gamma$ in $\F_q$.   
\end{thm}

\begin{proof}
Suppose $n$ is an integer such that in the transformed basis given by $J$, for an initial vector $\displaystyle \begin{pmatrix}
x_1 \\
x_0
\end{pmatrix}$ in this basis we have 
\begin{equation*}
D^n\begin{pmatrix}
x_1 \\
x_0 
\end{pmatrix} = \begin{pmatrix}
x_1 \\
x_0 
\end{pmatrix}.    
\end{equation*} 
Then, we can write this as
\begin{equation*}
\begin{pmatrix}
\gamma^n -1 & n\gamma^{n-1} \\
0 & \gamma^n -1 
\end{pmatrix} \begin{pmatrix}
x_1 \\
x_0 
\end{pmatrix} = \begin{pmatrix}
0 \\
0 
\end{pmatrix}.      
\end{equation*} 
From this, we get the two conditions
\begin{equation}
\label{eqn: repeated-roots-condition1}
(\gamma^n - 1)x_1 + n\gamma^{n-1}x_0 = 0,  
\end{equation}
and 
\begin{equation}
\label{eqn: repeated-roots-condition2}
(\gamma^{n}-1)x_0 = 0.
\end{equation} 

We need to consider two cases.
\begin{case}
If $x_0$ is equal to zero,  then Equation \ref{eqn: repeated-roots-condition1} reduces to $(\gamma^n - 1)x_1 = 0$. Since $x_1\neq 0$ for this case (otherwise, we would have the set of  zero initial conditions), the orbit length must be the smallest $n$ such that $|\gamma|$ divides $n$. Hence,  the orbit length is $|\gamma|$. 
\end{case}

\begin{case}
If $x_0$ is not equal to zero,  then Equation \ref{eqn: repeated-roots-condition2} implies that $\gamma^n - 1 =0$. Hence, Equation \ref{eqn: repeated-roots-condition1} reduces to $n\gamma^{n-1}x_0 = 0$. Since $x_0\neq 0$, we conclude that $n\gamma^{n-1} = 0$ i.e. $p \mid n$ where $p$ is the characteristic of $\F_q$. So, the orbit length $n$ is the smallest positive integer that is divisible by $p$ and $|\gamma|$. Since $|\gamma|$ must divide $q-1$, we have $\gcd(p,|\gamma|)=1$. Hence, $n$ must be equal to the $\lcm(p, |\gamma|) = p|\gamma|$. 
\end{case}
\end{proof}

\begin{Remark}
Note that since $p \nmid q-1$, we can conclude that $pl \nmid q-1$. Therefore, Corollary \ref{cor: length-divides-q-1} on the divisibility of $q-1$ by all possible orbit lengths as stated in Section \ref{sec: preliminaries}, does not carry over to the repeated roots case.  
\end{Remark}

Next, we consider the number of orbits of each type. In order to analyze this, we require the following two lemmas.
\begin{lem}
\label{lem: repeated-roots-orbits-of-lenght-l}
Every orbit containing an element $\displaystyle\begin{pmatrix}
x \\
0 
\end{pmatrix}$ for $x \neq 0$  is of the form 
\[
\left\{ \begin{pmatrix}
\gamma^kx \\
0
\end{pmatrix} \,\, \mid \,\, k\in \mathbb{N}\right\}
\]and of length $l$.      
\end{lem}
\begin{proof}
The lemma follows from the fact that 
\begin{eqnarray*}
\begin{pmatrix}
\gamma & 1 \\
0 & \gamma 
\end{pmatrix}^k\begin{pmatrix}
x \\
0 
\end{pmatrix} &=&  \begin{pmatrix}
\gamma^k & k\gamma^{k-1} \\
0 & \gamma^k 
\end{pmatrix}\begin{pmatrix}
x \\
0 
\end{pmatrix} \cr
& & \cr 
&=& 
\begin{pmatrix}
\gamma^kx \\
0
\end{pmatrix},
\end{eqnarray*} 
which has length $l =|\gamma|$.    
\end{proof}

\begin{lem}
\label{lem: repeated-roots-orbits-of-lenght-pl}
Any orbit containing a point of the form $\displaystyle \begin{pmatrix}
x_1 \\
x_0
\end{pmatrix}
$ where $x_0 \neq 0$, has length $pl$. 
\end{lem} 

\begin{proof}
Assume the orbit has length $l = |\gamma|$. Then 
\begin{eqnarray*}
\begin{pmatrix}
x_1 \\
x_0
\end{pmatrix} &=& \begin{pmatrix}
\gamma & 1 \\
0 & \gamma 
\end{pmatrix}^l\begin{pmatrix}
x_1 \\
x_0
\end{pmatrix} \cr
& & \cr  
&=& \begin{pmatrix}
\gamma^lx_1 + l\gamma^{l-1}x_0\\
\gamma^l x_0
\end{pmatrix}. 
\end{eqnarray*}
Therefore, $\gamma^lx_1 + l\gamma^{l-1}x_0 = x_1 $ and so $l\gamma^{l-1}x_0 = 0$. Since $p \nmid l$, we have $x_0 =0$ a contradiction. 
\end{proof}

\begin{thm}
\label{thm: repeated-roots-number-of-orbits}
Let $\F_q$ be a finite field of characteristic of $p$. Suppose that the characteristic polynomial of $Q$ has a repeated root $\gamma$ over $\F_q$ of order $l$. Then there are $\displaystyle \frac{q-1}{l}$ orbits of length $l$ and $\displaystyle \frac{q(q-1)}{pl}$ orbits of length $pl$. 
\end{thm}

\begin{proof}
From Lemmas \ref{lem: repeated-roots-orbits-of-lenght-l} and \ref{lem: repeated-roots-orbits-of-lenght-pl}, the orbits of length $l$ contain only non-zero elements of the form  $\displaystyle\begin{pmatrix}
x \\
0
\end{pmatrix}$ and vice-versa. Since there are $q-1$ such elements, we conclude that the number of orbits of length $l$ is equal to $\displaystyle \frac{q-1}{l}$. All other non-zero elements belong to orbits of length $pl$. Hence, there are 
\[
\frac{q^2 -1 - (q-1)}{pl} = \frac{q(q-1)}{pl}
\]
such orbits. 
\end{proof}

\begin{Remark}
Note that in Theorem \ref{thm: repeated-roots-number-of-orbits}, if we consider the case when $\F_q$ is a finite field of prime order i.e. $q$ is a prime, then there are an \emph{equal} number of non-trivial orbits of each type. 
\end{Remark}

\begin{Example}
\label{ex: permutation-even-char}
Let $\F_q$ be of characteristic $2$, and consider the case when $b=1$. Then $Q$ has repeated roots if and only if $Q$ is the permutation matrix $\displaystyle \begin{pmatrix} 0 & 1 \\ 1 & 0 \end{pmatrix}$. 

To see this, if $Q$ has a repeated root then from Equation \ref{eqn: coeff-formulas-1} we have $a = 2 \gamma = 0$. On the other hand, if $a=0$ and $b=1$ then from Equation \ref{eqn: coeff-formulas-2} we get $\gamma^2 = 1$ so that $\gamma = 1$ is a repeated root. By Theorem \ref{thm: repeated-roots-number-of-orbits} there are $q-1$ orbits of length $1$ and $\displaystyle \frac{q(q-1)}{2}$ orbits of length $2$.      
\end{Example}

\begin{Example}
Let $Q = \begin{pmatrix} 2 & -1 \\ 1 & 0 \end{pmatrix}$ in $\F_3$. Then its characteristic polynomial $x^2 - 2x + 1$, splits into repeated roots $\gamma_1 = \gamma_2 = 1$. Any element of the form  $\displaystyle \begin{pmatrix}
x_0 \\
x_0
\end{pmatrix}$ is a fixed point. Hence, there are two non-trivial orbits of length $1$ and two non-trivial orbits of length $3$.  
\end{Example}

\begin{Example} Consider the matrix $Q = \begin{pmatrix} 8 & -3 \\ 1 & 0 \end{pmatrix}$ over $\mathbb{F}_{13}$. Then $\gamma =4$ is a repeated root. Since the order of $\gamma$ is $6$ in $\mathbb{F}_{13}$, we have $2$ orbits of length $6$ and $2$ orbits of length $78$. 
\end{Example}

\begin{Example}
Consider the field $\F_{25}$. Let $x$ denote a generator of $\F_{25}^\times$ and define the matrix
\[
Q = \begin{pmatrix} 2x + 2 & -(x^2 + 2x + 1) \\ 1 & 0 \end{pmatrix}. 
\] 
The characteristic polynomial of $Q$ factors into the repeated root $\gamma= x + 1$.  Since the order of $\gamma$ is $12$ in $\mathbb{F}_{25}^\times$, we have an orbit structure of $2$ orbits of length $12$ and $10$ orbits of length $60$.
\end{Example}

When $b=1$ the case of repeated roots occurs only if $q \equiv 1 \mod 4$, and if $\F_q$ is of some odd characteristic $p$ we have orbit lengths of $4$ and $4p$. We state this in Proposition \ref{prop: repeated-roots-q-is-1mod4}.

\begin{prop}
\label{prop: repeated-roots-q-is-1mod4}
Suppose that $\displaystyle Q = \begin{pmatrix} a & 1 \\ 1 & 0 \end{pmatrix}$ over $\F_q$ of odd characteristic. If $Q$ has repeated roots over $\F_q$ then $q \equiv 1 \mod 4$. In this case, we have orbits of length $4$ and $4p$. 
\end{prop}

\begin{proof}
If $\gamma$ is a repeated root of $Q$, then from Equation \ref{eqn: coeff-formulas-1} we have $\displaystyle \gamma = \frac{a}{2}$. Therefore, using Equation \ref{eqn: coeff-formulas-2} we get $\displaystyle \gamma^2 =  \frac{a^2}{4} = -1$ so that $\gamma^4 = 1$. Hence, the order of $\gamma$ divides $4$. Since $\gamma^2 = -1$, we cannot have $\gamma$ be equal to $1$ or $-1$ and so the order of $\gamma$ cannot be $1$ or $2$ and hence must be $4$. Therefore, $4 \mid q-1$ and we conclude that $q \equiv 1 \mod 4$. By Theorem \ref{thm: repeated-roots-orbit-lengths}, we have orbits of length $4$ and $4p$.       
\end{proof}

\begin{Example}
Consider the standard Fibonacci matrix $Q = \begin{pmatrix} 1 & 1 \\ 1 & 0 \end{pmatrix}$ over $\mathbb{F}_{5}$. There is exactly one non-trivial orbit of length $4$ and one non-trivial orbit of length $20$. 
\end{Example}

\section{Lucas Primitive Roots}
\label{sec: Lucas-Primitive-Roots}

In this section we consider $\displaystyle Q = \begin{pmatrix} a & 1 \\ 1 & 0 \end{pmatrix}$, which corresponds to the special case of sequences of the form $x_{n+1} = ax_n + x_{n-1}$ in $\F_q$. Recall from Section \ref{sec: Introduction} that  we define a root $\gamma$ of the characteristic polynomial $p(x) = x^2 - ax -1$ to be a Lucas primitive root (LPR) if $\gamma$ is a generator of $\F_q^\times$. When $a$ also equals one, we have a Fibonacci sequence and such a root is called a Fibonacci primitive root (FPR) in the literature. Shanks in \cite{MR0297695} studies FPRs$\mod p$. In particular, it is shown that if $p(x)$ has FPRs then if $p \equiv 1 \mod 4$ there are two FPRs and if $p \equiv 3 \mod 4$ and $p \neq 5$ then $p(x)$ has only one FPR. In Propositions \ref{prop: q-is-3-mod-4} and \ref{prop: q-is-1-mod-4}, we provide similar results for the existence of LPRs over any finite field $\F_q$.  

\begin{lem}
\label{lem: gamme-gives-a}
Given $\displaystyle Q = \begin{pmatrix} a & 1 \\ 1 & 0 \end{pmatrix}$ , then $\gamma$ is a root if and only if $a = \gamma - \gamma^{-1}$. Hence, if $\gamma$ is a generator of $\F_q^\times$ then $\gamma$ is also an LPR for $Q$ with $a = \gamma - \gamma^{-1}$. 
\end{lem}

\begin{proof}
This follows from the fact that $\gamma$ is a root iff $\gamma^2 - a\gamma - 1 =0$ iff  $a = \gamma - \gamma^{-1}$ (multiplying the quadratic equation by $\gamma^{-1}$). 
\end{proof}

\begin{prop} 
\label{prop: q-is-3-mod-4}
Let $q$ be such that $q \equiv 3 \mod 4$, and hence $q = 2s + 1$ where $2 \nmid s$. Consider $\displaystyle Q = \begin{pmatrix} a & 1 \\ 1 & 0 \end{pmatrix}$ over $\F_q$. Then the following hold
\begin{enumerate}[(a)]
\item If $\gamma_1 \in \F_q$ is a root of $Q$ and $|\gamma_1| =s$, then its conjugate $\gamma_2$ is an LPR. 
\item There are exactly $\phi(s)$ values of $a$ such that $Q$ has one LPR, where $\phi$ is Euler's function. And this occurs when the conjugate root has order $s$. Moreover, it is not possible to find a $Q$ with distinct roots both of which are LPRs. 
\end{enumerate}
\end{prop}

\begin{proof}\mbox{}
\begin{enumerate}[(a)]
\item Since $q \equiv 3 \mod 4$, $\F_q$ is of odd characteristic. Moreover, by Proposition \ref{prop: repeated-roots-q-is-1mod4} $Q$ cannot have repeated roots and so it must have distinct roots $\gamma_1$ and $\gamma_2$ over $\F_q$. Now if $Q$ has a root $\gamma_1$ of odd order $s$, then the order of $\gamma_2$ is $2s$ by Corollary \ref{cor: r-is-prime}. Hence, the conjugate root $\gamma_2$ is an LPR. 
\item Note that $\F_q^\times$ is isomorphic to $\Z_{q-1}$, and so the number of generators of $\F_q^\times$ is equal to $\phi(q-1)$. Now $\phi(q-1) = \phi(2s) = \phi(2)\phi(s) = \phi(s)$. Hence, the number of generators of $\F_q^\times$ is equal to the number of elements of order $s$ in $\F_q^\times$. By Lemma \ref{lem: gamme-gives-a} to each element $\gamma_1$ in $\F_q^\times$ of order $s$, we can associate a matrix $Q$ with root $\gamma_1$ such that $a = \gamma_1 - \gamma_1^{-1}$. Using part (a), $\gamma_1$ has a conjugate root $\gamma_2$ of order $2s$. Hence, such a matrix $Q$ has exactly one LPR. And those are all the possible matrices $Q$ with LPRs, since there are only $\phi(s)$ generators to $\F_q^\times$. In particular, there is no matrix $Q$ for which both roots are LPRs. 
\end{enumerate}
\end{proof}

In the special case when $q=2p+1$ for some odd prime $p$ (for example, $p$ could be a Sophie Germain prime), we can quite easily construct all possible matrices $Q$ which have an LPR. This is stated in Corollary \ref{cor: 2p-plus-1-LPRs}.

\begin{cor}
\label{cor: 2p-plus-1-LPRs}
Suppose $\F_q$ is such that $q=2p + 1$ where $p$ is an odd prime. Let $\gamma$ be any element of $\F_q^\times$ such that $\gamma \neq 1$ or $-1$. Then $\displaystyle Q = \begin{pmatrix} a & 1 \\ 1 & 0 \end{pmatrix}$ with $a = \gamma - \gamma^{-1}$ has exactly one LPR. 
\end{cor}

\begin{proof}
Since $\gamma \neq 1$ or $-1$, the order of $\gamma$ is not equal $1$ or $2$. Hence, the order of $\gamma$ is either $p$ or $2p$. In either case, using Proposition \ref{prop: q-is-3-mod-4} $Q$ has exactly one LPR.     
\end{proof}

Phong in \cite{MR1089524} studies LPRs of sequences $x_{n+1} = ax_n + x_{n-1}$ mod $q^e$ with initial conditions of $x_0 = 0, x_1 =1$, for a prime $q$ and integer $e$. In particular, he derives a version of Corollary \ref{cor: 2p-plus-1-LPRs} for the particular case when $e=1$ (see Corollary 3 of \cite{MR1089524}). As demonstrated in Example \ref{example: LPRs-F27}, our result is valid over any finite field satisfying the conditions of Corollary \ref{cor: 2p-plus-1-LPRs}.

\begin{Example}
Consider the finite field $\F_7$, where $q = 7 = 2 \times 3 + 1$. If we take $\gamma_1 = 2$,  then its conjugate is $\gamma_2= -\gamma_1^{-1} = 3$ and $a = 5$.  For $\gamma_1 = 4$ we have $\gamma_2 = -\gamma_1^{-1} = 5$ and $a = 2$. In each case, using Corollary \ref{cor: 2p-plus-1-LPRs} the order of $\gamma_1$ is $3$ and $\gamma_2$ is $6$. By Theorem \ref{thm: distinct-roots-orbit-numbers}, we have $2$ orbits of length $3$ and $7$ orbits of length $6$ in both cases. 
\end{Example}



\begin{Example}
\label{example: LPRs-F27}
Consider the finite field $\F_3[x]/(x^3-x+1)$ of order $27$. Note that $q= 2\times 13+1$, and hence we have $\phi(13) = 12$ values of $\gamma$ such that the conjugates are LPRs to $Q$ with $a =\gamma - \gamma^{-1}$. We list all the elements $\gamma\neq \pm 1$, the conjugates $-\gamma^{-1}$  and the corresponding $a$ for this field in Table \ref{table: LPRs-F27}. For each $a$ given in Table \ref{table: LPRs-F27}, by Corollary \ref{cor: 2p-plus-1-LPRs} the associated matrix $Q$ has exactly one LPR of order $26$ and another root of order $13$. By Theorem \ref{thm: distinct-roots-orbit-numbers}, there are $2$ orbits of length $13$ and $27$ orbits of length $26$.

\begin{table}[h!]
\centering
\begin{tabular}{||c | c||} 
\hline
$\gamma$, $-\gamma^{-1}$ & $a = \gamma - \gamma^{-1}$ \\ 
\hline 
$x, x^2 + 2$ & $x^2 + x + 2$ \\
$x^2, x^2 + x + 2$ & $2x^2 + x + 2$ \\
$x + 2, x^2 + x$ & $x^2 + 2x + 2$ \\
$x^2 + 2x, x + 1$ & $x^2 + 1$ \\
$2x^2 + x + 2, 2x^2 + 2$ & $x^2 + x + 1$ \\
$x^2 + x + 1, x^2 + 2x + 2$ & $2x^2$ \\
$2x, 2x^2 + 1$ & $2x^2 + 2x + 1$ \\
$2x^2, 2x^2 + 2x + 1$ & $x^2 + 2x + 1$ \\
$2x + 1, 2x^2 + 2x$ & $2x^2 + x + 1$ \\
$2x^2 + x, 2x + 2$ & $2x^2 + 2$ \\
$x^2 + 2x + 1, x^2 + 1$ & $2x^2 + 2x + 2$ \\
$2x^2 + 2x + 2, 2x^2 + x + 1$ & $x^2$ \\
\hline
\end{tabular} \\
\caption{Values of $\gamma$ and $a$ for $\F_q = \F_3[x]/(x^3-x+1)$.}
\label{table: LPRs-F27}
\end{table}
\end{Example}

In Proposition \ref{prop: q-is-1-mod-4}, we look at the case when $q \equiv 1 \mod 4$. 
In this case we may have a repeated root. We show that $Q$ either has two LPRs or none. 
\begin{prop}
\label{prop: q-is-1-mod-4}
Suppose $q \equiv 1 \mod 4$, and hence $q = 2^ts + 1$ where $t > 1$ and $2 \nmid s$. Let $\displaystyle Q = \begin{pmatrix} a & 1 \\ 1 & 0 \end{pmatrix}$ and suppose its characteristic polynomial splits over $\F_q$. Then either $Q$ has two LPRs or none. In the latter case, the orders of both roots of $Q$ are bounded by $2s$. 
\end{prop}

\begin{proof}
Note that by Proposition \ref{prop: repeated-roots-q-is-1mod4}, if $\gamma$ is a repeated root of $Q$ then it would have order $4$ and $\gamma$ is not an LPR. In the case when $Q$ has distinct roots and both are of even order, their orders must be equal by Corollary \ref{cor: r-is-prime}. Hence, if any one root is an LPR then so is the other. On the other hand, if one root is of odd order then the order of the conjugate root is twice this by Corollary \ref{cor: r-is-prime}. In this case, if $\gamma$ is the root of odd order then its order must divide $s$. Hence, the orders of both roots are bounded by $2s < q-1$ and neither root is an LPR.  
\end{proof}

\begin{Example}
Consider the matrix $\displaystyle Q = \begin{pmatrix} x+1 & 1 \\ 1 & 0 \end{pmatrix}$ over the field $\F_5[x]/(x^2-2)$. The characteristic equation of $Q$ has distinct roots $4x + 2 + (x^2-2)$ and $2x + 4 + (x^2-2)$, both with an order of $24$.   
As an example, with the initial values of $x_0 = x + 1$ and $x_1 = 4x + 1$ we get a resulting orbit as follows:

\begin{align*}
\{ &\ x + 1, \ 4x + 1, \ x, \ 3, \ 4x + 3, \ 2x + 4, \ 1, \ 3x, \ 3x + 2, \ 3x + 3, 4x + 1, \ 3x + 2, \\
&\ \ 4x + 4, \ x + 4, \ 4x, \ 2, \ x + 2, \ 3x + 1, \ 4, \ 2x, 2x + 3, \ 2x + 2, \ x + 4, \ 2x + 3 \,\, \}. 
\end{align*}

\end{Example}

\section{Irreducible over $\mathbb{F}_q$}
\label{sec: Irreducible-case}

In this section, we consider the case when the characteristic polynomial $p(x)$ of $Q$ remains irreducible over $\F_q$, and hence it splits over a quadratic extension $\F_{q^2}$.

Recall that if $p(x)$ is irreducible in $\F_q[x]$, then its splits into distinct roots $\gamma$ and $\gamma^q$ over an appropriate extension field $\F_{q^2}$. Moreover, the $q$-power map $x \mapsto x^q$ in $\F_{q^2}$ permutes the roots of $p(x)$ (see for example \cite{Conrad-finite-fields}). Using this fact, we show in Theorem \ref{thm: orbit-lengths-irreducible-case} that the roots of $p(x)$ in $\F_{q^2}$ have the same order and hence applying Lemma \ref{lem: distinct-roots-possible-lengths} we conclude that all orbits of $G$ have the same length. This approach is essentially the one taken in \cite{MR2910306}. However, the authors in \cite{MR2910306} restrict their theory to second order sequences in $\F_p$ for a prime $p$ with initial conditions of $x_0 =0, x_1 =1$, whereas we state our theorem in a more general setting.  

\begin{thm}
\label{thm: orbit-lengths-irreducible-case}
Let $\F_q$ be a finite field, and let $\displaystyle Q = \begin{pmatrix} a & b \\ 1 & 0 \end{pmatrix}$ be such that its characteristic polynomial $p(x)$ is irreducible over $\F_q$. Then the lengths of the non-trivial orbits of $G$ under its canonical action on $\F_q \times \F_q$ are of equal length.
\end{thm}

\begin{proof}
Let $\gamma_1$ and $\gamma_2$ be the roots of $p(x)$ over a quadratic extension $\F_{q^2}$ of $\F_q$. Suppose that $m = \left|\gamma_1\right|$ and $n = \left|\gamma_2\right|$. Then, $\displaystyle \gamma_1^n = (\gamma_2^q)^n =(\gamma_2^n)^q=1$, so $m \mid n$. By a similar argument $n \mid m$, and hence $m=n$. Then applying Lemma \ref{lem: distinct-roots-possible-lengths} we conclude that all non-trivial orbits have equal length. 
\end{proof}


\begin{cor}
\label{cor: number-of-orbits-nonresidue-case}
The number of non-trivial orbits is equal to $\displaystyle \frac{q^2-1}{l}$, where $l$ is the length of each orbit.
\end{cor}

\begin{proof} 
This follows from the fact that in this case there is only one non-trivial orbit length of $l = |\gamma_1| = |\gamma_2|$. 
\end{proof}

Following a similar argument to the bound given in \cite{MR2910306} in the setting of a finite field $\F_p$ for $p$ a prime (see Theorem 8 in \cite{MR2910306}), we have the following upper bound on the 
orbit lengths. 

\begin{prop}
\label{prop: order-bound-nonresidue-case}
The orbit lengths of $G$ are bounded from above by $\displaystyle 2(q+1)|b^2|$. 
\end{prop}

\begin{proof}
Let $p(x)$ have roots $\gamma_1, \gamma_2$ in $\F_{q^2}$. Then for any one of its roots say $\gamma_1$, we have 
\begin{eqnarray*}
\gamma_1^{2(q+1)} &=& (\gamma_1^q)^2\gamma_1^2 \cr
                &=& \gamma_2^2 \gamma_1^2 \cr
                &=& b^2. 
\end{eqnarray*}
Hence, $\gamma_1^{2(q+1)|b^2|} =1$. Form this, we conclude that the order of $\gamma_1$ divides $2(q+1)|b^2|$ and the inequality follows. 
\end{proof}

Proposition \ref{prop: order-bound-nonresidue-case} gives us a lower bound on the number of orbits, as stated in the next corollary. 
\begin{cor}
The number of non-trivial orbits is greater than or equal to $\displaystyle \frac{q-1}{2|b^2|}$.  
\end{cor}
\begin{proof}
This follows from Corollary \ref{cor: number-of-orbits-nonresidue-case} and Proposition \ref{prop: order-bound-nonresidue-case}. 
\end{proof}

\begin{Example}
Consider the matrix $\displaystyle Q = \begin{pmatrix} 1 & 3 \\ 1 & 0 \end{pmatrix}$ over $\F_5$. Then the discriminant $\Delta = a^2 + 4b = 3$ is a quadratic non-residue in $\F_5$. The order of $\displaystyle \gamma = \frac{1 + \sqrt{3}}{2}$ in $\F(\sqrt{3})$ is $24$, which is the upper bound given by Proposition \ref{prop: order-bound-nonresidue-case}. There is exactly one non-trivial orbit of length $24$ in this case. 
\end{Example}

\begin{Example} 
Consider $\displaystyle Q = \begin{pmatrix} 1 & \sqrt{3} \\ 1 & 0 \end{pmatrix}$ over $\F_5(\sqrt{3})$. Then the discriminant $\Delta = a^2 + 4b = 1 + 4\sqrt{3}$ is a quadratic non-residue in $\F_5(\sqrt{3})$. For, if it is a quadratic residue then we must have an $\alpha + \beta\sqrt{3}$ in $\F_5(\sqrt{3})$ such that $(\alpha + \beta\sqrt{3})^2 = 1 + 4\sqrt{3}$. This gives us a set of two equations mod $5$
\begin{eqnarray*}
\alpha \beta = 2, \cr
\alpha^2 + 3 \beta^2 = 1,
\end{eqnarray*}
which has no solution. Hence, by Theorem \ref{thm: orbit-lengths-irreducible-case} we have non-trivial orbits of all equal length. Using Sage Math, we find that there are $3$ non-trivial orbits of length $208$, which is the upper bound on the orbit length as given by Proposition \ref{prop: order-bound-nonresidue-case}. 
\end{Example}

\bibliographystyle{siam}
\bibliography{References.bib}

\end{document}